\documentclass[11pt,letterpaper]{amsart}

\usepackage{mathrsfs,color,balance,bm,amsmath,amsfonts,amssymb,dsfont,amscd,extarrows,enumerate,verbatim}
\usepackage[all]{xy}
\numberwithin{equation}{section}

\newtheorem{thm}{Theorem}[section]
\newtheorem{cor}[thm]{Corollary}
\newtheorem{lem}[thm]{Lemma}

\theoremstyle{definition}
\newtheorem{remark}[thm]{Remark}

\newcommand{\mr}{\mathbb{R}}

\newcommand{\mc}{\mathbb{C}}

\newcommand{\rw}{\rightarrow}

\newcommand{\nve}{\vec{\bm n}}

\DeclareMathOperator{\diam}{diam}

\DeclareMathOperator{\rank}{rk}

\makeatletter
\def\@settitle{\begin{center}%
  \baselineskip14\p@\relax
  \bfseries
  \uppercasenonmath\@title
  \@title
  \ifx\@subtitle\@empty\else
     \\[1ex]\uppercasenonmath\@subtitle
     \footnotesize\mdseries\@subtitle
  \fi
  \end{center}%
}
\def\subtitle#1{\gdef\@subtitle{#1}}
\def\@subtitle{}
\makeatother

\usepackage[pagebackref=true,colorlinks=true,bookmarksopen,bookmarksnumbered,citecolor=red, linkcolor=blue, urlcolor=cyan]{hyperref}

\allowdisplaybreaks[4]
\begin{document}

\title[Quantitative Carleman-type estimates]
{Quantitative Carleman-type estimates for holomorphic sections over bounded domains}

\author[X. Qin]{Xiangsen Qin}
\address{Xiangsen Qin: \ Chern Institute of Mathematics and LPMC, Nankai University \\ Tianjin 300071, China}
\email{qinxiangsen@nankai.edu.cn}

\begin{abstract}
     This paper establishes quantitative Carleman-type inequalities for holomorphic sections of Hermitian vector bundles over bounded domains in $\mathbb{C}^n$ with $n \geq 2$. We first prove a Sobolev-type inequality with explicit constants for the Laplace operator, which leads to quantitative Carleman-type estimates for holomorphic functions. These results are then extended to holomorphic sections of Hermitian vector bundles satisfying certain curvature restrictions, yielding quantitative versions where previously only non-quantitative forms were available. The proofs refine existing methods through careful constant tracking and by estimating the radius of the uniform sphere condition of the boundary through the Lipschitz constant of its outward unit normal vector.
\end{abstract}

\maketitle
\tableofcontents
    \section{Introduction}
    In his celebrated paper \cite{C21}, Carleman established a beautiful proof of the two-dimensional isoperimetric inequality by proving 
    the following estimate:
\begin{equation}\label{equ:main}
\int_{\mathbb{D}^2}|f|^2\,dV \leq \frac{1}{4\pi}\left(\int_{\partial\mathbb{D}^2}|f|\,dS\right)^2,
\end{equation}
 for any \( f \in C^0(\overline{\mathbb{D}^2}) \) that is holomorphic in \( \mathbb{D}^2 \), where \( \mathbb{D}^2 \subset \mathbb{C} \) denotes the unit disk.  
Aronszajn \cite{A50} later extended \eqref{equ:main} to simply connected domains with analytic boundary, and Jacobs \cite{J72} further treated multiply connected domains. Some generalizations to $L^p$-norms have also been obtained; see, for example, \cite[Theorem 19.9]{K84} and \cite{MP84}, and other references.

In a different direction, Hang--Wang--Yan \cite{HWY07} proved that
\begin{equation}\label{equ:main1}
\|f\|_{L^{\frac{2n}{n-2}}(\mathbb{D}^n)} \leq n^{\frac{2-n}{2n}} \omega_n^{\frac{2-n}{2n(n-1)}} \|f\|_{L^{\frac{2(n-1)}{n-2}}(\partial\mathbb{D}^n)},
\end{equation}
where $f$ is a smooth harmonic function \( f \) on \( \overline{\mathbb{D}^n} \),  \( \mathbb{D}^n \subset \mathbb{R}^n \ (n \geq 3) \) is the unit ball,
and $\omega_n$ is the surface area of the unit sphere in $\mr^n$.
A natural and interesting question is whether an inequality of the form \eqref{equ:main1} can be established for general bounded domains and for general \( L^p \)-norms.  
The absence of a systematic treatment of this generalization in the literature forms the primary motivation for this work\\
\indent To streamline the subsequent presentation, we define a key notation. Let $\Omega \subset \mathbb{R}^n$ be a bounded $C^2$ domain with outward unit normal vector $\nve$. We set
$$\operatorname{LC}_\Omega:=\inf\left\{L\geq 0|\ |\nve(x)-\nve(y)|\leq L|x-y|,\ \forall x,y\in\partial\Omega\right\},$$
$$\operatorname{LD}_\Omega:=\left\{\begin{array}{ll}
 0, & \text{ if }\Omega\text{ is convex},\\
 \diam(\Omega)\cdot\operatorname{LC}_\Omega, &\text{ otherwise}.
\end{array}\right.$$

\indent To construct Carleman-type inequalities for holomorphic functions, we first prove the following Sobolev-type inequality for the 
Laplace operator $\Delta$. 
\begin{thm}\label{thm:laplace}
Let $\Omega\subset\mr^n(n\geq 3)$ be a bounded domain with $C^2$-boundary.
For  $1<p<\infty$, set 
$$p^*:=\frac{np}{n-1},\ p^{\sharp}:=\frac{np}{n+2p-1}.$$
Then, for every  $f\in C^2(\Omega)\cap C^0(\overline\Omega)$, there are constants $\delta_1:=\delta_1(n,p)>0$,
$\delta_2:=\delta_2(n,p,\operatorname{LD}_\Omega)>0$ such that 
\begin{equation}\label{equ:laplace}
\|f\|_{L^{p^*}(\Omega)}\leq \delta_1\|\Delta f\|_{L^{p^{\sharp}}(\Omega)}+\delta_2\|f\|_{L^p(\partial\Omega)}.
\end{equation}
Moreover, the constants $\delta_1$ and $\delta_2$ can be explicitly given by 
\[\delta_1 = \frac{2\omega_n^{-\frac{2}{n}}}{n-2} 
        \left( \frac{6^n p^\sharp}{3(p^\sharp - 1)} \right)^{1 - \frac{2p^\sharp}{n}} 
        \left( \frac{p^\sharp - 1}{n - 2p^\sharp} \right)^{\frac{2(p^\sharp - 1)}{n}},
  \]
\[
  \delta_2 = 2p^{-\frac{1}{np}} (p-1)^{\frac{1-n}{np}}
        \left( \frac{8^n \cdot n^{\frac{5}{2}} \cdot \omega_{n-1}}{(n-1)(2^n - 4)\omega_n^{1+1/n}} 
        \cdot\max\{8, \operatorname{LD}_\Omega\} \right)^{\frac{1}{p}}.
  \]
\end{thm}

 In their seminal work \cite{CM16}, Cianchi and  Maz'ya established the existence of a constant  $\delta:=\delta(n,p)>0$ such that 
 \begin{equation}\label{equ:heassian}
 \delta\|f\|_{L^{p^*}(\Omega)}\leq \|\nabla^2 f\|_{L^{p^{\sharp}}(\Omega)}+\|f\|_{L^p(\partial\Omega)},\ \forall f\in C^2(\Omega)\cap C^0(\overline\Omega),
 \end{equation} 
 where $\nabla^2f$ denotes the Hessian of $f$. Thus, Inequality (\ref{equ:laplace}) holds for compactly supported $f$, modulo constants.
  Nevertheless, the general case appears not to be amenable to the techniques used in \cite{CM16}. 
 On the other hand, one need to note that a non-quantitative version of Theorem \ref{thm:laplace} has been presented in \cite[Theorem 1.13]{DHQ24}.
It should be noted that, by \cite[Theorem 3.24]{GGS10}, one cannot in general expect the constant $\delta_2$ in inequality (\ref{equ:laplace}) to depend solely on $n$ and $p$.\\
\indent We now outline the main ideas for proving Theorems \ref{thm:laplace} . Since corresponding non-quantitative versions have been established in \cite{DHQ24} and \cite{DHQ25}, our approach refines their methodology through careful constant tracking. A key aspect is controlling the radius of the uniform sphere condition satisfied by $\partial\Omega$ via the Lipschitz constant $\operatorname{LC}_\Omega$, an idea inspired by the work of \cite{LP20}.\\

 \indent An immediate consequence of Theorem \ref{thm:laplace} is the following corollary, obtained through approximation of convex domains by smooth convex domain
 (see \cite[Lemma 3.2.3.2]{G11}).
\begin{cor}\label{thm:convex}
Let $\Omega\subset\mr^n(n\geq 3)$ be a bounded convex domain.
For any $1<p<\infty$, set 
$$p^*:=\frac{np}{n-1},\ p^{\sharp}:=\frac{np}{n+2p-1}.$$
Then for any  $f\in C^2(\Omega)\cap C^0(\overline\Omega)$, there are constants $\delta_1:=\delta_1(n,p)>0$,
$\delta_2:=\delta_2(n,p)>0$ such that 
\begin{equation}\label{equ:convex}
\|f\|_{L^{p^*}(\Omega)}\leq \delta_1\|\Delta f\|_{L^{p^{\sharp}}(\Omega)}+\delta_2\|f\|_{L^p(\partial\Omega)}.
\end{equation}
Moreover, the constants $\delta_1$ and $\delta_2$ can be explicitly given by 
\[\delta_1 = \frac{2\omega_n^{-\frac{2}{n}}}{n-2} 
        \left( \frac{6^n p^\sharp}{3(p^\sharp - 1)} \right)^{1 - \frac{2p^\sharp}{n}} 
        \left( \frac{p^\sharp - 1}{n - 2p^\sharp} \right)^{\frac{2(p^\sharp - 1)}{n}},
  \]
\[
  \delta_2 = 2p^{-\frac{1}{np}} (p-1)^{\frac{1-n}{np}}
        \left( \frac{8^{n+1} \cdot n^{\frac{5}{2}} \cdot \omega_{n-1}}{(n-1)(2^n - 4)\omega_n^{1+1/n}}  \right)^{\frac{1}{p}}.
  \]
\end{cor}

 \indent The harmonicity of holomorphic functions yields the following application of Theorem \ref{thm:laplace} and Corollary \ref{thm:convex}:
\begin{cor}\label{cor:holomorphic}
Let $\Omega\subset\mc^n(n\geq 2)$ be either a bounded domain with $C^2$-boundary or a bounded convex domain.
For any $1<p<\infty$, set 
$$p^*:=\frac{2np}{2n-1}.$$
Then for any $f\in C^0(\overline\Omega)\cap\mathcal{O}(\Omega)$, there is a constant $\delta:=\delta(\operatorname{LD}_\Omega)>0$
such that 
\begin{equation}\label{equ:holomorphic}
\|f\|_{L^{p^*}(\Omega)}\leq 2p^{-\frac{1}{2np}} (p-1)^{\frac{1-2n}{2np}}
        \left( \frac{64^n \cdot (2n)^{\frac{5}{2}} \cdot \omega_{2n-1}\cdot\delta}{(2n-1)(4^n - 4)\omega_{2n}^{1+1/(2n)}} 
         \right)^{\frac{1}{p}}\|f\|_{L^p(\partial\Omega)}.
\end{equation}
Furthermore, the constant $\delta$ can be taken to be 
\[\delta= \left\{\begin{array}{ll}
 8, & \text{ if }\Omega\text{ is convex},\\
 \max\{8, \operatorname{LD}_\Omega\} , &\text{ otherwise}.
\end{array}\right.
\]
\end{cor}
\indent Finally, we further extend Corollary \ref{cor:holomorphic} to the setting of holomorphic sections of Hermitian vector bundles:
\begin{thm}\label{thm:section}
Let $\Omega \subset \mc^n$ ($n \geq 2$) be a bounded domain with smooth boundary, and let $E$ be a Hermitian holomorphic vector bundle defined in a neighborhood of $\overline{\Omega}$. Fix an integer $r$ with $0 \leq r \leq n$.
Suppose the curvature of $E$ satisfies the following bounds:
\begin{itemize}
    \item $\mathfrak{Ric}_{r,0}^E \geq -K$, where $K$ is a constant such that
    \[
    0 \leq K \leq \frac{1}{2} j_{n-1}^2 \left( \frac{\omega_{2n-1}}{2n} \right)^{\frac{1}{n}} |\Omega|^{-\frac{1}{n}},
    \]
    \item $\mathfrak{Ric}_{r,1}^E \geq -K_+$ for some constant $K_+ \geq 0$.
\end{itemize}
Here, $j_\nu$ denotes the first positive root of the Bessel function $J_\nu$ of the first kind of degree $\nu \in \mr$.

For any $1 < p < \infty$,  set $p^* := \frac{2np}{2n-1}$. Then, for any 
$
f \in C^0(\overline{\Omega}, \Lambda^{r,0} T^*\mc^n \otimes E) \cap \mathcal{O}(\Omega, \Lambda^{r,0} T^*\mc^n \otimes E),
$
\begin{equation*}
\|f\|_{L^{p^*}(\Omega)} \leq 2 (p-1)^{\frac{1-2n}{2np}} p^{-\frac{1}{2np}} (2n \omega_{2n}^{1-\frac{1}{2n}} C_3)^{\frac{1}{p}} \big( e^{K \cdot \diam(\Omega)} \big)^{1-\frac{1}{p}} \|f\|_{L^p(\partial\Omega)},
\end{equation*}
with the constant $C_3$ as given in Theorem \ref{thm:Green form}, depending on $n$, $K$, $K_+$, $\diam(\Omega)$, and $\operatorname{LD}_\Omega$.
\end{thm}
 
A non-quantitative version of Theorem \ref{thm:section} was established in \cite[Corollary 1.4]{DHQ25}
 using Green forms estimates under the assumption that $K = 0$, and the proof therein can be readily adapted to this case by using Lemma \ref{lem: estimate Green}. When $K \neq 0$, however, the approach in \cite{DHQ25} is not directly applicable.
Instead, we employ some techniques from \cite{LZZ21} to establish the following estimates for the Green form: 
\begin{thm}\label{thm:Green form}
Let $\Omega \subset \mc^n$ ($n \geq 2$) be a bounded domain with smooth boundary, and let $E$ be a Hermitian holomorphic vector bundle defined in a neighborhood of $\overline{\Omega}$. Fix  $r,s\in\{0,\cdots,n\}$.
Suppose the curvature of $E$ satisfies the following bounds:
\begin{itemize}
    \item $\mathfrak{Ric}_{r,s}^E \geq -K$, where $K$ is a constant such that
    \[
    0 \leq K \leq \frac{1}{2} j_{n-1}^2 \left( \frac{\omega_{2n-1}}{2n} \right)^{\frac{1}{n}} |\Omega|^{-\frac{1}{n}},
    \]
    \item $\mathfrak{Ric}_{r,s+1}^E \geq -K_+$ for some constant $K_+ \geq 0$.
    
    \item $\mathfrak{Ric}_{r,s-1}^E \geq -K_+$ for some constant $K_{-} \geq 0$.
\end{itemize}

 Let $G_{r,s}^E(\cdot,\cdot)$ be the Schwarz kernel of the Dirichlet Green operator for $\square_{r,s}^E$ on $\Omega$,
then it satisfies the following estimates:  
\begin{itemize}
  \item[(i)] There is a constant $C_1:=C_1(n,K,\diam(\Omega))> 0$ such that 
  for all $(x,y)\in\Omega\times\Omega$, 
  $$|G_{r,s}^E(x,y)|\leq C_1|x-y|^{2-2n},$$
  where  
\[ C_1 =\frac{1}{(2n-2)\omega_{2n}}\cdot\left(K^{\frac{n-1}{2}}+\exp\left(2^{2n+11}(n\omega_{2n})^{\frac{1}{n-1}}\diam(\Omega)^{\frac{4n-2}{n-1}}K^{\frac{1}{2}}\right)\right).\]
    \item[(ii)] There is a constant $C_2:=C_2(n,K,\diam(\Omega),\operatorname{LD}_\Omega)> 0$ such that 
  for all $(x,y)\in\Omega\times\Omega$, 
  $$|G_{r,s}^E(x,y)|\leq C_2|x-y|^{1-2n}\delta(y),$$
  where  
\[ C_2 =\left(K^{\frac{n}{2}}+\exp\left(2^{n+8}(n\omega_{2n})^{\frac{1}{n}}\diam(\Omega)^{4}K^{\frac{1}{2}}\right)\right)\cdot\frac{ 2^{4n-2}\cdot\max\{8,\operatorname{LD}_\Omega\}}{(4^n-4)\omega_{2n} }.\]

  \item[(iii)] There is a constant $C_3:=C_3(n,K,K_{+},K_{-},\diam(\Omega),\operatorname{LD}_\Omega)> 0$ such that 
  for all $(x,y)\in\Omega\times\Omega$, 
  $$|\bar\partial_yG_{r,s}^E(x,y)|+|\bar\partial_y^*G_{r,s}^E(x,y)|\leq C_3|x-y|^{1-2n},$$
  where
\[ C_3 = 4^n\max\{C_1,C_2\}\cdot \sqrt{32C_{11}},\]
and the constant 
$$C_{11}=2^{3n^2+9n+6}\omega_{2n}\left(\max\{K_+,K_{-}\}\cdot\diam(\Omega)^2 + 1\right)^{n}\left(1 + K\diam(\Omega)^2\right)$$
is defined in Lemma \ref{lem:important estimate}.
\end{itemize}  
\end{thm}
\leavevmode\\
\subsection*{Acknowledgements}
    		The author is grateful to Professor Fusheng Deng and Professor Xiaonan Ma for valuable discussions and suggestions on related topics.
     
\section{Notations and conventions}
Throughout this work, we adopt the following notational conventions.
Set $\mr_+:=\{x\in\mr|\ x>0\}$, and set $\mathbb{N}:=\{0,1,\cdots\}$. The surface area of the unit sphere of $\mr^n$ is denoted by 
$\omega_n$. For any $r>0$, and $x\in\mr^n$, set 
$$B(x,r):=\{y\in\mr^n|\ |x-y|<r\}.$$
For any Lebesgue measurable subset $A$ of $\mr^n$, we use $|A|$ to denote the Lebesgue measure of $A.$
The Laplace operator on $\mr^n$ is denoted by $\Delta$. Let $\Omega$ denote a bounded Lipschitz domain in $\mathbb{R}^n$. The diameter of $\Omega$ is written as $\diam(\Omega)$.  The Lebesgue measure on $\Omega$ is denoted by $dV$, while the Hausdorff measure on the boundary $\partial\Omega$ is denoted by $dS$.

We now introduce several function spaces used throughout this paper. Let \( k \in \mathbb{N} \cup \{\infty\} \) and \(\Omega \subset \mathbb{R}^n\) be an open subset. We denote by \( C^k(\Omega) \) the space of complex-valued \(C^k\)-smooth functions on \(\Omega\), and by \( C^k_c(\Omega) \) the subspace of \(C^k(\Omega)\) consisting of functions with compact support. The space of continuous functions on the closure \(\overline{\Omega}\) is written as \( C^0(\overline{\Omega}) \). When \(\Omega \subset \mathbb{C}^n\), the space of holomorphic functions on \(\Omega\) is denoted by \(\mathcal{O}(\Omega)\). For \( 1 \leq p\leq \infty \), the spaces \( L^p(\Omega) \) and \( L^p(\partial\Omega) \) consist of \(L^p\)-integrable functions on \(\Omega\) and \(\partial\Omega\), respectively, endowed with the norms \(\|\cdot\|_{L^p(\Omega)}\) and \(\|\cdot\|_{L^p(\partial\Omega)}\).

Let $\Omega\subset\mc^n$ be an open subset, and let $E$ be a Hermitian holomorphic vector bundle $E$ defined in a neighborhood of $\overline{\Omega}$.
  We denote the Hermitian inner product of $E$ by $\langle\cdot,\cdot\rangle$ and the corresponding fibre norm by $|\cdot|$. The space of holomorphic sections of $E$ over $\Omega$ is denoted by $\mathcal{O}(\Omega,E)$, while $C^k(\Omega,E)$ and $C^0(\overline{\Omega},E)$ represent the spaces of $C^k$-smooth sections over $\Omega$ and continuous sections over $\overline{\Omega}$, respectively, where $k\in\mathbb{N}\cup\{\infty\}$.  
   Let\( 1 \leq p\leq \infty \). The $L^p$-norm of a Lebesgue measurable section $s$ of $E$ is defined as the $L^p$-norm of the function $|s|.$
   The spaces \( L^p(\Omega,E) \) and \( L^p(\partial\Omega,E) \) then consist of \(L^p\)-integrable sections on \(\Omega\) and \(\partial\Omega\), respectively. For any integers $r,s\in\{0,\cdots,n\}$, let $\Lambda^{r,s}T^*\mc^n$ be the bundle of smooth $(r,s)$-forms on $\mc^n$.
   Denote by $\bar\partial$ the dbar operator on $\Lambda^{r,s}T^*\mc^n\otimes E$ and by $\bar\partial^*$ its formal adjoint. The $\bar\partial$-Laplacian on $
   \Lambda^{r,s}T^*\mc^n\otimes E$  is defined as $\square_{r,s}^E=\bar\partial^*\bar\partial+\bar\partial\bar\partial^*$. A section $f\in C^2(\Omega,\Lambda^{r,s}T^*\mc^n\otimes E)$ is said to be harmonic if it satisfies $\square_{r,s}^Ef=0$. We use $\nabla$ to denote the Chern connection of $\Lambda^{r,s}T^*\mc^n\otimes E$, and use $\nabla^*$ to denote its formal adjoint. The Weitzenböck curvature operator on $\Lambda^{r,s}T^*\mc^n\otimes E$ is defined by 
   $$\mathfrak{Ric}_{r,s}^E:=2\square_{r,s}^E-\nabla^*\nabla.$$
   For the local expression of $\mathfrak{Ric}_{r,s}^E$, please see \cite[Theorem 3.1]{L10}.
  This paper will repeatedly employ the following Bochner–Weitzenböck formula:
  for any $f\in C^2(\Omega,\Lambda^{r,s}T^*\mc^n\otimes E)$,
  $$\Delta\frac{|f|^2}{2}=\operatorname{Re}\langle (-2\square_{r,s}^E+\mathfrak{Ric}_{r,s}^E)f,f\rangle+|\nabla f|^2.$$
  
\section{Some useful lemmas}\ 
 In this section, we collect several lemmas that will be used later.\\
 \indent The following estimate for the Riesz potential is standard. For completeness and to provide an explicit constant, we include a proof here.
 \begin{lem}\label{lem:better estimate}
    Let $\Omega\subset\mr^n(n\geq 2)$ be an open subset.
    For any $0<a<n$ and any $1<p<\infty$ such that $n>pa,$ there is a constant $C:=C(n,p,a)>0$ such that 
     $$\|I_af\|_{L^{\frac{np}{n-pa}}(\Omega)}\leq C\|f\|_{L^p(\Omega)},\ \forall \ f\in L^p(\Omega),$$
   where 
   $$I_af(x):=\int_{\Omega}\frac{|f(y)|}{|x-y|^{n-a}}dV(y),\ \forall x\in \Omega,$$
   and $C$ can be taken to be 
   $$C=2\omega_n^{1-\frac{a}{n}}
     \left(\frac{6^np}{(2^a-1)(p-1)}\right)^{1-\frac{pa}{n}}
     \left(\frac{p-1}{n-pa}\right)^{\frac{(p-1)a}{n}}.$$
    \end{lem}     
 
   \begin{proof}
    For any $x\in \Omega$ and $r>0$, we get 
    \begin{align*}
   \int_{B(x,r)}\frac{|f(y)|}{|x-y|^{n-a}}dV(y)&=\sum_{k=0}^\infty\int_{B(x,2^{-k}r)\setminus B(x,2^{-k-1}r)}\frac{|f(y)|}{|x-a|^{n-a}}dV(y)\\
    &\leq \sum_{k=0}^\infty|B(x,2^{-k}r)|\cdot (2^{k+1}r)^{a-n} \mathcal{M}f(x)\\
    &=\frac{2^n\omega_n}{2^a-1}\mathcal{M}f(x)r^a,
    \end{align*}
    where $\mathcal{M}f$ is the Hardy-Littlewood maximal function of $f$.
    By H\"older's inequality, 
    \begin{align*}
     \int_{\Omega\setminus B(x,r)}\frac{|f(y)|}{|x-y|^{n-a}}dV(y)&\leq \|f\|_{L^p(\Omega)}\left(\int_{\Omega\setminus B(x,r)}\frac{1}{|x-y|^{\frac{p(n-a)}{p-1}}}dV(y)\right)^{\frac{p-1}{p}}\\
    &\leq \|f\|_{L^p(\Omega)}\left(\omega_n\int_{r}^\infty s^{\frac{p(a-n)}{p-1}+n-1}ds\right)^{\frac{p-1}{p}}\\
    &=\left(\frac{(p-1)\omega_n}{n-pa}\right)^{\frac{p-1}{p}}\|f\|_{L^p(\Omega)}r^{a-\frac{n}{p}}.
    \end{align*}
    For any $x\in \Omega$, choose $r>0$ such that 
    $$\frac{2^n\omega_n}{2^a-1}\mathcal{M}f(x)r^a=\left(\frac{(p-1)\omega_n}{n-pa}\right)^{\frac{p-1}{p}}\|f\|_{L^p(\Omega)}r^{a-\frac{n}{p}},$$
    then 
    $$I_af(x)\leq C_0\mathcal{M}f(x)^{1-\frac{pa}{n}}\|f\|_{L^p(\Omega)}^{\frac{pa}{n}}, $$
    where 
    $$C_0:=2\left(\frac{2^n\omega_n}{2^a-1}\right)^{1-\frac{pa}{n}}\left(\frac{(p-1)\omega_n}{n-pa}\right)^{\frac{(p-1)a}{n}}.$$
    By the $L^p$-boundedness of the operator $\mathcal{M}$ (see \cite[Theorem 3.2.7]{H19}), 
    \begin{align*}
      \|I_af\|_{L^{\frac{np}{n-pa}}(\Omega)}\leq C
      \|\mathcal{M}f\|_{L^p(\Omega)}^{\frac{n-pa}{n}}\|f\|_{L^p(\Omega)}^{\frac{pa}{n}}
      \leq C_0\left(\frac{3^np}{p-1}\right)^{\frac{n-pa}{n}}\|f\|_{L^p(\Omega)}.
    \end{align*}
   \end{proof}
   
   \begin{remark}
   In \cite[Theorem 4.3]{LL01}, a different constant $C$ is given, and an optimal constant 
   $C$ is also given when $p=(2n)/(n+a)$.
   \end{remark}

 To introduce the  next lemma, we first recall the definition of {\it weak $L^p$ spaces}. Let $(X,\mu)$ be a measure space and $f$ be a measurable function on $X$.
  For any $1\leq p\leq \infty$, the weak $L^p$-norm of $f$ is defined by 
    $$\|f\|_{L^{p,\infty}(X,\mu)}:=\sup_{t>0}t\mu\left(\{x\in X|\ |f(x)|>t\}\right)^{\frac{1}{p}},$$
    and we say $f$ belongs to the space $L^{p,\infty}(X,\mu)$ if and only if $\|f\|_{L^{p,\infty}(X,\mu)}<\infty$.
    For $1\leq p\leq \infty$, the space of $L^p$-integrable functions on $X$ is denoted by $L^p(X,\mu)$, with the correspond norm written as
    $\|\cdot\|_{L^p(X,\mu)}$. We shall omit explicit reference to the measure $\mu$ when it is the Lebesgue measure  \\
    \indent The following weak-type estimate is a consequence of \cite[Lemma 5.3]{CM16} and its proof.
 \begin{lem}\label{lem:weak}
    Let $\Omega\subset\mr^n(n\geq 2)$ be a bounded Lipschitz domain, then for any $f\in C^0(\overline\Omega)$,
    $$\|Jf\|_{L^{\frac{n}{n-1},\infty}(\Omega)}\leq n\omega_n^{1-\frac{1}{n}}\|f\|_{L^1(\partial \Omega)},\ \forall f\in C^0(\overline\Omega),$$
    where 
    $$Jf(x):=\int_{\partial \Omega}\frac{|f(y)|}{|x-y|^{n-1}}dS(y),\ \forall x\in \Omega.$$
    \end{lem}
    We conclude this section by stating  a version of the Marcinkiewicz interpolation theorem with explicit constants, which will be needed in the subsequent analysis. For a detailed proof, we refer the reader to \cite[Theorem 6.28]{F99}.
    \begin{lem}\label{lem:interpolation}
     Let $(X,\mu)$ and $(Y,\nu)$ be two $\sigma$-finite measure spaces, and denote by
     $\mathcal{M}(Y,\nu)$ the space of  $\nu$-measurable functions on $Y$.
     Let 
     $$T\colon L^{1}(X,\mu)+L^{\infty}(X,\mu)\to \mathcal{M}(Y,\nu)$$
      be a sublinear operator.
     Fix a parameter $n\in(1,\infty)$ and  suppose there are constants $C_1,C_2>0$ such that 
    $$\|Tf\|_{L^{\frac{n}{n-1},\infty}(Y,\nu)}\leq C_1\|f\|_{L^{1}(X,\mu)},\ \forall f\in L^{1}(X,\mu),$$
    $$\|Tf\|_{L^{\infty}(Y,\nu)}\leq C_2\|f\|_{L^{\infty}(X,\mu)},\ \forall f\in L^{\infty}(X,\mu),$$
    then for all $1<p<\infty$, the operator $T$ admits a bounded extension from  
    $L^p(X,\mu)$ to $L^{(np)/(n-1)}(Y,\nu)$ satisfying
    $$\|Tf\|_{L^{\frac{np}{n-1}}(Y,\nu)}\leq C\|f\|_{L^p(X,\mu)},\ \forall f\in L^p(X,\mu),$$
    where the constant $C$ is given explicitly by 
    $$C=2(p-1)^{\frac{1-n}{np}}p^{-\frac{1}{np}}C_1^{\frac{1}{p}}C_2^{1-\frac{1}{p}}.$$ 
   \end{lem}

\section{Sobolev-type inequalities for the Laplace operator $\Delta$}
 We present the proof of Theorem \ref{thm:laplace} in this section, relying on the following estimates for the Green function. The proof of these estimates adapts the argument in \cite[Theorem 3.2]{GW82}. In particular, we implement the idea from \cite{LP20} of controlling the radius in the uniform sphere condition for $\partial\Omega$ via the Lipschitz constant $\operatorname{LC}_\Omega$.
 
\begin{lem}\label{lem: estimate Green}
Let $\Omega\subset\mr^n$ be a bounded domain with $C^2$-boundary.
Let $G(\cdot,\cdot)$ be the negative Dirichlet Green function of 
$\Omega$ with respect to the Laplace operator $\Delta$, then it satisfies the following estimates:
\begin{itemize}
  \item[(i)] For any $(x,y)\in \Omega\times\Omega$, 
  $$|G(x,y)|\leq \frac{|x-y|^{2-n}}{(n-2)\omega_n}$$
  \item[(ii)] For any $(x,y)\in \Omega\times\Omega$, 
  $$|G(x,y)|\leq \frac{2^{2n-2}\cdot\max\{8,\operatorname{LD}_\Omega\}}{(2^n-4)\omega_n }\delta(x)|x-y|^{1-n},$$
  where $\delta(x):=\inf_{z\in\partial\Omega}|x-z|$.
  \item[(iii)] For any $(x,y)\in \Omega\times\overline\Omega$, 
  $$|\nabla_yG(x,y)|\leq \frac{8^n n^{\frac{3}{2}}\omega_{n-1}}{(n-1)(2^n-4)\omega_n^2}\cdot\max\{8,\operatorname{LD}_\Omega\}|x-y|^{1-n},$$
  where $\nabla$ denotes the gradient operator.
\end{itemize}
\end{lem}
 \begin{proof}
(i) This is a standard result following from the maximum principle.\\
(ii) 
 Set
 $$r_0:=\left\{\begin{array}{ll}
 \infty, & \text{ if }\Omega\text{ is convex},\\
 \frac{1}{\operatorname{LC}_\Omega}, &\text{ otherwise}.
\end{array}\right.$$
By \cite[Theorem 1]{LP20}, $\Omega$ satisfies the uniform exterior sphere condition of radius $r_0$, i.e. 
for any $0<b<\infty$,  any $0<r\leq \min\{b,r_0\}$ and any $x\in\partial\Omega$, there exist $z\in\mr^n$ such that 
$$B(z,r)\subset\mr^n\setminus\overline\Omega,\ |z-x|=r.$$
\indent  Fix $x,y\in\Omega$ with $x\neq y$, we consider two cases. \\
{\bf Case 1:} $\delta(x)<\min\left\{\frac{|x-y|}{8},r_0\right\}$.\\
    \indent Set $r:=\min\left\{\frac{|x-y|}{8},r_0\right\}$. Choose $z_x\in\partial\Omega$ and choose $x^*\in\mr^n\setminus\overline{\Omega}$
     such that 
     $$|x-z_x|=\delta(x),\ |x^*-z_x|=r,\ B(x^*,r)\subset\mr^n\setminus\overline{\Omega}.$$
     Define 
    $$u(z):=
      \frac{2^n}{4-2^n}\left[\left(\frac{r}{|z-x^*|}\right)^{n-2}-1\right],\ \forall z\in \mr^n\setminus\{x^*\},$$
    then $u$ satisfies
    $$
    \left\{\begin{array}{ll}
      \Delta u=0&\text{ in }\mr^n\setminus\{x^*\},\\
       u=0 & \text{ on }\partial B(x^*,r),\\
       u=1 & \text{ on }\partial B(x^*,2r).
    \end{array}\right.
    $$ Moreover, 
    $$\sup_{z\in\mr^n\setminus B(x^*,r)}|\nabla u(z)|\leq 
      \frac{2^n(n-2)}{(2^n-4)r}.$$
    Since $z_x\in\partial B(x^*,r)$, the mean value theorem gives 
    $$u(x)=|u(x)-u(z_x)|\leq \frac{2^n(n-2)\delta(x)}{(2^n-4)r}.$$
    For any $z\in\partial B(x^*,2r)\cap\Omega$, we have 
    $$|z_x-z|\leq |z_x-x^*|+|x^*-z|=3r,\ |x-z|\leq |x-z_x|+|z_x-z|\leq 4r,$$
    $$|z-y|\geq |x-y|-|x-z|\geq |x-y|-4r\geq \frac{|x-y|}{2}.$$
    By (i), 
    $$
    |G(z,y)|\leq \frac{|z-y|^{2-n}}{(n-2)\omega_n}\leq \frac{2^{n-2}|x-y|^{2-n}}{(n-2)\omega_n }u(z).
    $$
    Note that $G(\cdot,y)|_{\partial\Omega}=0$. Applying the maximum principle in
    $\Omega\cap (B(x^*,2r)\setminus \overline{B(x^*,r)})\ni x$, we obtain
    $$|G(x,y)|\leq \frac{2^{n-2}|x-y|^{2-n}u(x)}{(n-2)\omega_n }\leq \frac{2^{2n-2}|x-y|^{2-n}\delta(x)}{(2^n-4)\omega_nr}.$$
    Hence,
    \begin{equation}\label{gradient:eequ1}
     |G(x,y)|\leq \frac{2^{2n-2}}{(2^n-4)\omega_n}\max\{8,\operatorname{LD}_\Omega\}|x-y|^{1-n}\delta(x). 
    \end{equation}
    {\bf Case 2:} $\delta(x)\geq \min\left\{r_0,\frac{|x-y|}{8}\right\}$.\\
    \indent In this case, 
    $$\frac{|x-y|}{\delta(x)}\leq \max\{8,\operatorname{LD}_\Omega\},$$
    so by (i),
    \begin{equation}\label{gradient:eequ2}
    |G(x,y)|\leq \frac{|x-y|^{2-n}}{(n-2)\omega_n}\leq \frac{1}{(n-2)\omega_n}\max\{8,\operatorname{LD}_\Omega\}|x-y|^{1-n}\delta(x).
    \end{equation}
    \indent  Combining {\bf Case 1} and {\bf Case 2}, i.e., Inequalities (\ref{gradient:eequ1}) and (\ref{gradient:eequ2}),
    we conclude that (ii) holds.\\
    (iii) Fix $x,y\in\overline\Omega$ with $x\neq y$. By continuity, we may assume $x,y\in\Omega$.
    Again, we consider two cases.\\
    {\bf Case 1:} $\delta(y)\leq |x-y|$.\\
    \indent In this case, we know $G(x,\cdot)$ is a harmonic function in $B\left(y,\frac{1}{2}\delta(y)\right)$. By the gradient estimate for 
    harmonic functions (see \cite[Chapter 4, Section 4, Problem 8]{J82}), 
    $$|\nabla_y G(x,y)|\leq \frac{4\sqrt{n}\gamma_n}{\delta(y)}\sup_{B\left(y,\frac{1}{2}\delta(y)\right)}|G(x,\cdot)|,$$
    where 
    $$\gamma_n:=\frac{n\omega_{n-1}}{(n-1)\omega_n}.$$
    For any $z\in B\left(y,\frac{1}{2}\delta(y)\right)$, 
    $$|x-z|\geq |x-y|-|y-z|\geq |x-y|-\frac{1}{2}\delta(y)\geq \frac{1}{2}|x-y|,$$
    $$\delta(z)\leq \delta(y)+|y-z|\leq 2\delta(y).$$ 
    By (ii),  
    \begin{align*}
    |G(x,z)|&\leq \frac{2^{2n-2}}{(2^n-4)\omega_n}\max\{8,\operatorname{LD}_\Omega\}|x-z|^{1-n}\delta(z)\\
         &\leq \frac{2^{3n-2}}{(2^n-4)\omega_n}\max\{8,\operatorname{LD}_\Omega\}|x-y|^{1-n}\delta(y).
    \end{align*}
    Therefore,
    \begin{equation}\label{gradient:eequ3}
    |\nabla_y G(x,y)|\leq \frac{\sqrt{n}2^{3n}\gamma_n}{(2^n-4)\omega_n}\max\{8,\operatorname{LD}_\Omega\}|x-y|^{1-n}.
    \end{equation}
    {\bf Case 2:} $\delta(y)>|x-y|$. \\
    \indent In this case, $G(x,\cdot)$ is a harmonic function in $B\left(y,\frac{1}{2}|x-y|\right)$, so 
    $$|\nabla_y G(x,y)|\leq \frac{4\sqrt{n}\gamma_n}{|x-y|}\sup_{B\left(y,\frac{1}{2}|x-y|\right)}|G(x,\cdot)|.$$
    For any $z\in B\left(y,\frac{1}{2}|x-y|\right)$, we have 
    $$|x-z|\geq |x-y|-|y-z|\geq \frac{1}{2}|x-y|,$$
    By (i), 
    $$|G(x,z)|\leq \frac{|x-z|^{2-n}}{(n-2)\omega_n}\leq \frac{2^{n-2}}{(n-2)\omega_n}|x-y|^{2-n}.$$
    Thus, 
    \begin{equation}\label{gradient:eequ4}
    |\nabla_y G(x,y)|\leq \frac{2^n\sqrt{n}\gamma_n}{(n-2)\omega_n}|x-y|^{1-n}.
    \end{equation}
    \indent Combining Inequalities (\ref{gradient:eequ3}) and (\ref{gradient:eequ4}), we obtain (iii).
 \end{proof}
 
 We now proceed to prove Theorem \ref{thm:laplace}. Under the same assumptions and notation as in the theorem.
 Let $G(\cdot,\cdot)$ be the negative Dirichlet Green function of $\Omega$.  By the Green representation formula, 
 for any $f\in C^2(\overline\Omega)$, and any $x\in\Omega$,
 $$f(x)=\int_{\Omega}G(x,y)\Delta f(y)dV(y)+\int_{\partial\Omega}\frac{\partial G(x,y)}{\partial\nve_y}dS(y).$$
 Using the Minkowski inequality and a smooth approximation argument, it suffices to establish the following estimates:

 \begin{lem}
The following estimates hold:
\begin{itemize}
  \item[(i)] \leavevmode For any $g \in L^{p^\sharp}(\Omega)$, there is a constant 
  $C:=C(n,p)>0$ such that 
  \[
  \|B_\Omega g\|_{L^{p^*}(\Omega)} \leq C \|g\|_{L^{p^\sharp}(\Omega)},
  \]
  where
  \[
  B_\Omega g(x) := \int_{\Omega} G(x, y) g(y) \, dV(y), \quad \forall x \in \Omega,
  \]
  and
  \[
  C_1 = \frac{2\omega_n^{-\frac{2}{n}}}{n-2} 
        \left( \frac{6^n p^\sharp}{3(p^\sharp - 1)} \right)^{1 - \frac{2p^\sharp}{n}} 
        \left( \frac{p^\sharp - 1}{n - 2p^\sharp} \right)^{\frac{2(p^\sharp - 1)}{n}}.
  \]

  \item[(ii)] \leavevmode For any $g \in L^p(\partial\Omega)$, there is a constant 
  $C:=C(n,p,\operatorname{LD}_\Omega)>0$ such that 
  \[
  \|B_{\partial\Omega} g\|_{L^{p^*}(\Omega)} \leq C \|g\|_{L^p(\partial\Omega)},
  \]
  where
  \[
  B_{\partial\Omega} g(x) := \int_{\partial\Omega} \frac{\partial G(x, y)}{\partial \nve_y} g(y) \, dS(y), \quad \forall x \in \Omega,
  \]
  and
  \[
  C = 2p^{-\frac{1}{np}} (p-1)^{\frac{1-n}{np}} 
        \left( \frac{8^n \cdot n^{\frac{5}{2}} \cdot \omega_{n-1}}{(n-1)(2^n - 4)\omega_n^{1+1/n}} 
        \cdot\max\{8, \operatorname{LD}_\Omega\} \right)^{\frac{1}{p}}.
  \]
\end{itemize}
\end{lem}

\begin{proof}
(i) This follows directly from Lemma \ref{lem:better estimate} and Part (i) of Lemma \ref{lem: estimate Green}.\\
(ii) First, observe that for any $g\in L^{\infty}(\partial\Omega)$,  
$$\|B_{\partial\Omega} g\|_{L^\infty(\Omega)}\leq \|g\|_{L^\infty(\partial\Omega)}.$$
Moreover, by Lemma \ref{lem:weak} and Part (iii) of Lemma \ref{lem: estimate Green} , for any $g\in L^1(\partial\Omega)$, 
$$\|B_{\partial\Omega} g\|_{L^{\frac{n}{n-1},\infty}(\Omega)}\leq C\|g\|_{L^1(\partial\Omega)},$$
where 
$$C=\frac{8^nn^{\frac{5}{2}}\omega_{n-1}}{(n-1)(2^n-4)\omega_n^{1+1/n}}\cdot\max\{8,\operatorname{LD}_\Omega\}.$$
The desired estimate now follows from Lemma \ref{lem:interpolation}.
\end{proof}

\section{Quantitative Carleman-type estimates for holomorphic sections}

In this section, we present the proofs of Theorem \ref{thm:section} and \ref{thm:Green form}. We work throughout under the same assumptions and notations as in
Theorem \ref{thm:Green form}. In particular, we fix the following notations: let $G(x,y)$ denote the Dirichlet Green function of $\Omega$, and let $H_{r,s}^E(t,x,y)$ (resp. $H(t,x,y)$) be the Dirichlet heat kernel associated with the operator $\square_{r,s}^E$ (resp. $-\Delta$) on $\Omega$.\\
\indent We begin by establishing a maximum principle for harmonic sections:
\begin{lem}\label{lem:Linfty estimate}
For any harmonic section $f\in C^0(\overline\Omega,\Lambda^{r,s}T^*\mc^n\otimes E)$, 
$$\sup_{\Omega}|f|\leq e^{K\diam(\Omega)}\sup_{\partial\Omega}|f|.$$
\end{lem}
\begin{proof}
If $\mathfrak{Ric}_{r,s}^E\geq 0$, then the Bochner-Weitzenb\"ock formula implies
$$\Delta|f|^2\geq 0.$$
By the maximum principle for subharmonic functions, it follows that 
$$\sup_{\Omega}|f|^2\leq \sup_{\partial\Omega}|f|^2.$$
In the general case,  we may assume $0\in\Omega$. 
Let $L$ be the trivial line bundle on $\overline\Omega$ equipped with the Hermitian metric $h_z=e^{K|z|^2-K\diam(\Omega)^2}$,
so that $\mathfrak{Ric}_{r,s}^{L\otimes E}\geq 0$ in $\Omega$. Let $t$ be the canonical holomorphic frame of $L$.
Then we have
$$\sup_{z\in\Omega}|f(z)|^2\cdot h_z(t,t)\leq \sup_{z\in\partial\Omega}|f(z)|^2\cdot h_z(t,t),$$
i.e.
$$\sup_{z\in\Omega}|f(z)|^2 e^{K|z|^2-K\diam(\Omega)^2}\leq \sup_{z\in\Omega}|f(z)|^2 e^{K|z|^2-K\diam(\Omega)^2}.$$
This yields  the desired estimate.
\end{proof}
A direct consequence of Lemma \ref{lem:Linfty estimate} is the following:
\begin{cor}
For any harmonic section $f\in  C^0(\overline\Omega,\Lambda^{r,s}T^*\mc^n\otimes E)$
such that $f|_{\partial\Omega}=0$, then $f\equiv 0$ in $\Omega$. In particular, all eigenvalues of 
$\square_{r,s}^E$ are strictly positive.
\end{cor}
\indent The following lemma, which is inspired by \cite[Lemma 2.2]{LZZ21}, provides a $C^0$-estimate for eigenfunctions of 
$-\Delta$.
\begin{lem}\label{lem:eigensection estimate}
Suppose $\phi\in C^2(\overline\Omega)$ and $\lambda\geq 0$ satisfy
$$-\Delta\phi=\lambda\phi,\quad \phi|_{\partial\Omega}=0.$$
Then  
\begin{equation}
\sup_{\Omega}|\phi|^2\leq 2^{n^2+2n}\lambda^n\int_{\Omega}|\phi|^2dV.
\end{equation}
\end{lem}
\begin{proof}
By Bochner-Weitzenb\"ock formula, we obtain
$$\Delta\frac{|\phi|^2}{2}=\operatorname{Re}(\langle \Delta\phi,\phi\rangle)+|\nabla\phi|^2\geq -\lambda |\phi|^2,$$
where $\operatorname{Re}(\cdot)$ denotes the real part of a complex number.
Set $v:=|\phi|^2$, then 
$$\Delta v\geq -2\lambda v.$$
For $p\geq 1$, integration by parts yields 
$$\int_{\Omega}\frac{2p-1}{p^2}|\nabla v^p|^2dV=-\int_{\Omega}v^{2p-1}\Delta vdV
\leq 2\lambda \int_{\Omega}v^{2p}dV.$$
Combining this with the Sobolev inequality on $\mc^n$, we obtain
$$\left( \int_{\Omega} v^{2p\alpha}dV\right)^{\alpha} \leq pC_4\int_{\Omega} v^{2p}dV,$$
where 
$$\alpha:=\frac{n}{n-1},\quad  C_4:=8\lambda.$$
Now take $p=\alpha^{k-1}$ for $k=1,2,3\cdots,$ and iterate to derive 
\begin{align*}
\left( \int_{\Omega} v^{2\alpha^{k}} dV\right)^{1/\alpha^{k}} &\leq C_4^{\alpha^{-(k-1)}} \alpha^{(k-1)\alpha^{-(k-1)}} \left( \int_{\Omega} v^{2\alpha^{k-1}} dV\right)^{1/\alpha^{k-1}}\\
&\leq C_4^{\sum_{i=1}^k \alpha^{-(i-1)}}\prod_{j=1}^k \alpha^{(j-1)\alpha^{-(j-1)}}\int_{\Omega} vdV.
\end{align*}
The desired estimate follows by taking the limit as $k\rw \infty$.
\end{proof}
Let $0<\mu_1\leq \mu_2\leq \cdots$ denote all the Dirichlet eigenvalues of $-\Delta$ on $\Omega$,
and let $\phi_1,\phi_2,\cdots$ denote the corresponding eigenfunctions. Using Lemma \ref{lem:eigensection estimate}, we derive the following heat kernel estimate:
\begin{cor}\label{cor:heat kernel}
For any $(t,x,y)\in\mr_+\times\Omega\times\Omega$, 
$$|H(t,x,y)|\leq 2^{n^2+4n+1}n\diam(\Omega)^{2n}e^{-\frac{\mu_1t}{2}}t^{-n}.$$
\end{cor}
\begin{proof}
Fix $(t,x,y)\in\mr_+\times\Omega\times\Omega$.
By the  maximum principle, 
\begin{equation}\label{equ:heat kernel}
|H(t,x,y)|\leq \frac{1}{(4\pi t)^n}e^{-\frac{|x-y|^2}{4t}}.
\end{equation}
Clearly, 
$$\sup_{x>0}xe^{-\frac{x}{n}}=ne^{-1}.$$
Similar to the proof of \cite[Corollary 4.6]{DL82}, one easily obtains (
\begin{equation}\label{equ:eigenvalues estimates}
   \mu_k\geq 4\pi n e^{-1}|\Omega|^{-\frac{1}{n}}k^{\frac{1}{n}},\ \forall k\geq 1.
\end{equation}
Note that for any $c>0$, we have 
$$
\sum_{k=1}^\infty e^{-ck^{\frac{1}{n}}}\leq \int_0^{\infty}e^{-cz^{\frac{1}{n}}}dz\leq c^{-n}n!,
$$
and 
$$\gamma:=\sup_{z>0}e^{-\frac{z}{4}}z^n=e^{-n}(4n)^n,\quad \omega_{2n}=\frac{2\pi^n}{(n-1)!}.$$
By Lemma \ref{lem:eigensection estimate} and Inequality (\ref{equ:eigenvalues estimates}), we obtain 
\begin{align*}
&\quad |H(t,x,y)|\\
&\leq 2^{n^2+2n}\sum_{k=1}^\infty e^{-\mu_kt}\mu_k^n
\leq 2^{n^2+2n}\gamma e^{-\frac{\mu_1t}{2}}\sum_{k=1}^\infty e^{-\frac{\mu_kt}{4}}\\
&\leq 2^{n^2+4n}\pi^{-n}n!|\Omega|e^{-\frac{\mu_1t}{2}}t^{-n}
\leq 2^{n^2+4n+1}n\diam(\Omega)^{2n}e^{-\frac{\mu_1t}{2}}t^{-n}.
\end{align*}
 \end{proof}
 We now proceed to prove Part (i) and (ii) of Theorem \ref{thm:Green form}.\\

{\bf\noindent Proof of (i) of Theorem \ref{thm:Green form}:} Without loss of generality, we may assume $K\neq 0$.
By \cite[Theorem 1.3]{BK23}, 
$$2K\leq \mu_1.$$
Fix any $(x,y)\in\Omega\times\Omega, x\neq y$. By \cite[Theorem 4.3]{DL82}, for all $t>0$, 
\begin{equation}\label{equ:heat comparision}
|H_{r,s}^E(2t,x,y)|\leq e^{2Kt}|H(t,x,y)|,
\end{equation}
then by Corollary \ref{cor:heat kernel} and Inequality (\ref{equ:heat kernel}), 
for any $t_0>0$, we obtain 
\begin{align*}
 &\quad \frac{1}{2}|G_{r,s}^E(x,y)|\leq \int_0^\infty |H_{r,s}^E(2t,x,y)|dt\\
 &\leq e^{t_0}\int_{0}^{\frac{t_0}{2K}}\frac{1}{(4\pi t)^n}e^{-\frac{|x-y|^2}{4t}}dt
 +2^{n^2+4n+1}n\diam(\Omega)^{2n}\int_{\frac{t_0}{2K}}^\infty t^{-n}dt,\\
 &\leq e^{t_0}\frac{|x-y|^{2-2n}}{(2n-2)\omega_{2n}}+2^{n^2+4n+1}n\diam(\Omega)^{2n}\frac{(2K)^{n-1}}{(n-1)t_0^{n-1}}\\
 &\leq \left(\frac{e^{t_0}}{(2n-2)\omega_{2n}}+2^{n^2+5n}n\frac{\diam(\Omega)^{4n-2}K^{n-1}}{(n-1)t_0^{n-1}}\right)|x-y|^{2-2n}.
\end{align*}
Choose $t_0>0$ such that 
$$\frac{1}{(2n-2)\omega_{2n}}=2^{n^2+5n}n\frac{\diam(\Omega)^{4n-2}K^{\frac{n-1}{2}}}{(n-1)t_0^{n-1}},$$
then 
$$t_0\leq 2^{2n+11}(n\omega_{2n})^{\frac{1}{n-1}}\diam(\Omega)^{\frac{4n-2}{n-1}}K^{\frac{1}{2}}.$$
Thus, 
\begin{align*}
&\quad |G_{r,s}^E(x,y)|\\
&\leq \left(K^{\frac{n-1}{2}}+\exp\left(2^{2n+11}(n\omega_{2n})^{\frac{1}{n-1}}\diam(\Omega)^{\frac{4n-2}{n-1}}K^{\frac{1}{2}}\right)\right)
\frac{|x-y|^{2-2n}}{(n-1)\omega_{2n}}.
\end{align*}
\qed\\

{\bf\noindent Proof of (ii) of Theorem \ref{thm:Green form}:} Without loss of generality, we may assume $K\neq 0$.
Set 
$$C_5:=\frac{2^{4n-2}\cdot\max\{8,\operatorname{LD}_\Omega\}}{(4^n-4)\omega_{2n} },\quad  C_6:=\frac{2^{\frac{n^2}{2}+5n-2}\cdot\max\{8,\operatorname{LD}_\Omega\}\cdot \diam(\Omega)}{4^n-4 },$$
$$C_7:=\frac{2^{n^2+8n+1}(n+1)^{n+1}}{(4^n-4)en^{n-1}}\max\{8,\operatorname{LD}_\Omega\}\cdot\diam(\Omega)^{2n+1}.$$
Fix $x,y\in\Omega$. By the Green representation formula,  Part (ii) of Lemma \ref{lem: estimate Green}
and Lemma \ref{lem:eigensection estimate},
for all $k\geq 1$,
\begin{equation}\label{equ:heat kernel 2}
|\phi_k(y)|\leq C_5\mu_k\delta(y)\int_{\Omega}|\phi_k(z)|\cdot|z-y|^{1-2n}dV(z) \leq C_6\mu_k^{\frac{n}{2}+1}\delta(y).
\end{equation}
By Lemma \ref{lem:eigensection estimate} and Inequality (\ref{equ:heat kernel 2}), for all $t>0$,
\begin{align*}
|H(t,x,y)|&\leq \sum_{k=1}^\infty e^{-\mu_k t}|\phi_k(x)|\cdot|\phi_k(y)|\leq 2^{\frac{n^2}{2}+n}C_6\sum_{k=1}^{\infty}e^{-\mu_k t}\mu_k^{n+1}\delta(y)\\
&\leq 2^{\frac{n^2}{2}+3n+2}\frac{(n+1)^{n+1}n!}{e(n\pi)^n}C_6|\Omega|e^{-\frac{\mu_1t}{2}}t^{-n-1}\delta(y)\\
&= C_7e^{-\frac{\mu_1t}{2}}t^{-n-1}\delta(y).
\end{align*}
Then by Part (i) of Lemma \ref{lem: estimate Green}, Inequality (\ref{equ:heat comparision}), for any $t_0>0$,
\begin{align*}
\frac{1}{2}|G_{r,s}^E(x,y)|&\leq e^{t_0}\int_0^{\frac{t_0}{2K}}|H(t,x,y)|dt+\int_{\frac{t_0}{2K}}^\infty |H(t,x,y)|dt\\
&\leq e^{t_0}|G(x,y)|+\int_{\frac{t_0}{2K}}^\infty |H(t,x,y)|dt\\
&\leq C_5e^{t_0}|x-y|^{1-2n}\delta(y)+C_7\delta(y)\int_{\frac{t_0}{2K}}^{\infty}t^{-n-1}dt\\
&\leq \left(C_5e^{t_0}+\frac{C_7\diam(\Omega)^{2n-1}}{n}\frac{(2K)^{n}}{t_0^{n}}\right)|x-y|^{1-2n}\delta(y)\\
\end{align*}
Choose $t_0$ such  that 
$$C_5=\frac{C_7\diam(\Omega)^{2n-1}}{n}\frac{2^nK^{\frac{n}{2}}}{t_0^{n}},$$
then we get 
$$t_0\leq 2^{n+8}(n\omega_{2n})^{\frac{1}{n}}\diam(\Omega)^{4}K^{\frac{1}{2}}.$$
Therefore, 
\begin{align*}
&\quad |G_{r,s}^E(x,y)|\\
&\leq C_5\left(K^{\frac{n}{2}}+\exp\left(2^{n+8}(n\omega_{2n})^{\frac{1}{n}}\diam(\Omega)^{4}K^{\frac{1}{2}}\right)\right)|x-y|^{1-2n}\delta(y),
\end{align*}
and the proof is complete.

Our approach of the $C^1$-estimates of the Green form relies on the following key lemma, whose proof draws inspiration from \cite[Proposition 3.3]{LZZ21}.
\begin{lem}\label{lem:important estimate}
 Let $B(x_0,2\rho)\subset\Omega$ be an open ball. 
 Then there is a constant $C_{11}:=C_{11}(n,K,K_{+},K_{-},\diam(\Omega))>0$ such that for any harmonic section $\phi\in C^\infty(\overline{B(x_0,2\rho)},\Lambda^{r,s}T^*\mc^n\otimes E),$  
$$\sup_{B(x_0,\rho/2)} \left(|\bar\partial\phi|^2 + |\bar\partial^*\phi|^2\right) \leq \frac{C_{11}}{\rho^2}\sup_{B\left(x_0,2\rho\right)} |\phi|^2,$$
where
$$C_{11}=2^{3n^2+9n+6}\omega_{2n}\left(\max\{K_+,K_{-}\}\cdot\diam(\Omega)^2 + 1\right)^{n}\left(1 + K\diam(\Omega)^2\right).$$
\end{lem}
\begin{proof}

Let $v := |\bar\partial\phi|^2 + |\bar\partial^*\phi|^2$. By the Bochner-Weitzenb\"ock formula and our assumptions,
\[
\Delta v \geq -2K_0 v.
\]
For any $p \geq 1$, this implies
\begin{equation}\label{equ: v^p-1}
\Delta v^p \geq -2pK_0 v^p,
\end{equation}
where 
$$K_0:=\max\{K_+,K_{-}\}.$$
Fix $0 < \mu < \nu \leq 1$, and let $\psi \in C^\infty(B(x_0,\rho))$ be a function satisfying
\[
0 \leq \psi\leq 1,\quad \psi|_{B(x_0,\mu\rho)}\equiv 1,\quad \psi|_{\partial B(x_0,\nu\rho)}\equiv 0,\quad |\nabla\psi| \leq \frac{2}{(\nu - \mu)\rho}.
\]
Multiplying both sides of \eqref{equ: v^p-1} by $\psi^2 v^p$ and integrating over $B\left(x_0,\rho\right)$ yields
\begin{equation}\label{equ: v^p-1 2}
\int_{B\left(x_0,\rho\right)} \psi^2 v^p \Delta v^p \geq -2pK_0 \int_{B\left(x_0,\rho\right)} \psi^2 v^{2p},
\end{equation}
where the volume element is omitted for notational simplicity.
Integration by parts gives
\[
\int_{B\left(x_0,\rho\right)} \psi^2 v^p \Delta v^p = -\int_{B\left(x_0,\rho\right)} \left(|\nabla(\psi v^p)|^2 - v^{2p}|\nabla\psi|^2\right),
\]
from which we obtain
\begin{equation}\label{equ: v^p-1 3}
\int_{B\left(x_0,\rho\right)} |\nabla(\psi v^p)|^2 \leq 2pK_0 \int_{B\left(x_0,\rho\right)} \psi^2 v^{2p} + \int_{B\left(x_0,\rho\right)} v^{2p} |\nabla\psi|^2.
\end{equation}
By the properties of $\psi$, we have
\[
\int_{B(x_0,\nu\rho)} |\nabla(\psi v^p)|^2 \leq \left(2pK_0 + \frac{4}{(\nu - \mu)^2 \rho^2}\right) \int_{B(x_0,\nu \rho)} v^{2p}.
\]
Applying the well known Sobolev inequality on $\mc^n$, we get
\[
\left(\int_{B(x_0,\nu\rho)} (\psi v^p)^{2\alpha}\right)^{1/\alpha} \leq 4 \int_{B(x_0,\nu\rho)} |\nabla(\psi v^p)|^2,
\]
where $\alpha:=n/(n-1).$
This implies 
\begin{equation}\label{iteration 1}
\left(\int_{B(x_0,\mu \rho)} v^{2p\alpha}\right)^{1/\alpha} \leq 8p\left(K_0 + \frac{2}{(\nu - \mu)^2 \rho^2}\right) \int_{B(x_0,\nu \rho)} v^{2p}.
\end{equation}
We now perform Moser iteration. For any $0<\tau<1$, $0<\theta<\rho$, and   $k =0, 1,\cdots$, define
\[
p_k := \alpha^k,\ \mu_k := \frac{1}{2}+\frac{\tau}{2^{k+1}},\ \nu_k := \frac{1}{2}+\frac{\tau}{2^{k}},\ r_k:=\nu_k\theta.
\]
Applying inequality (\ref{iteration 1}) with $p = p_k$, $\mu = \mu_k$, $\nu = \nu_k$, we obtain
\[
\|v^2\|_{L^{\alpha^{k+1}}(B(x_0,r_{k+1}))} \leq \left[2^{2k+6}\alpha^k\left(\frac{K_0}{8}+ \frac{1}{\tau^2\rho^2}\right)\right]^{1/\alpha^k} \|v^2\|_{L^{\alpha^k}(B(x_0,r_k))}.
\]
Iterating this inequality and taking the limit as $k \to \infty$, we get
\begin{equation}\label{equ:moser}
\sup_{B(x_0,\theta/2)} v^2 \leq C_8 \int_{B\left(x_0,\theta/2+\tau\rho\right)} v^2,
\end{equation}
where
\[
C_8 = 2^{\frac{3\alpha}{(\alpha-1)^2}+\frac{6\alpha}{\alpha-1}} \left(\frac{K_0}{8} + \frac{1}{\tau^2\rho^2}\right)^{\frac{\alpha}{\alpha-1}}
=2^{3n^2+3n}\left(\frac{K_0}{8} + \frac{1}{\tau^2\rho^2}\right)^{n}.
\]

Next, for $j=0,1,\cdots$, set 
$$\theta_j:=\sum_{k=0}^j\frac{1}{2^{k}}\rho,\quad  \tau_j:=\frac{1}{2^{j+1}},$$
then from (\ref{equ:moser}), we have 
\begin{align*}
\sup_{B(x_0,\theta_j/2)}v^2&\leq C_9\cdot 4^{jn}\int_{B\left(x_0,\theta_{j+1}/2\right)} v^2\leq C_9\cdot 4^{jn}\left(\sup_{B(x_0,\theta_j/2)}v^2\right)^{\frac{1}{2}}
\cdot \int_{B\left(x_0,\rho\right)} v,
\end{align*}
where 
$$C_9=2^{3n^2+5n}\left(\frac{K_0}{32} + \frac{1}{\rho^2}\right)^{n}.$$
Iterating again, we obtain
\begin{equation*}
\sup_{B(x_0,\rho/2)}v^2\leq 16^nC_9^2\left(\int_{B\left(x_0,\rho\right)} v\right)^2,
\end{equation*}
i.e., 
\begin{equation}\label{equ:moser 1}
\sup_{B(x_0,\rho/2)}v\leq 4^nC_9\int_{B\left(x_0,\rho\right)} v.
\end{equation}

Now we bound the right-hand side of \eqref{equ:moser 1}. Let $\chi\in C_c^\infty(B(x_0,2\rho))$ be a cut-off function satisfying
\[
0 \leq \chi \leq 1,\quad \chi|_{B(x_0,\rho)} \equiv 1,\quad |\chi'| \leq \frac{2}{\rho}.
\]
Using integration by parts, Kato's inequality, and the Cauchy-Schwarz inequality, we obtain
\begin{align}\label{equ:parts}
\int_{B\left(x_0,2\rho\right)} \chi^2 \Delta|\phi|^2 &= -4\int_{B\left(x_0,2\rho\right)} \chi|\phi| \nabla|\phi| \cdot \nabla\chi \nonumber \\
&\leq 4\int_{B\left(x_0,2\rho\right)} \chi|\phi| \cdot |\nabla\phi| \cdot |\nabla\chi| \nonumber \\
&\leq \frac{3}{2}\int_{B\left(x_0,2\rho\right)} |\nabla\phi|^2 \chi^2 + \frac{8}{3}\int_{B\left(x_0,2\rho\right)} |\phi|^2 |\nabla\chi|^2.
\end{align}
The Bochner-Weitzenb\"ock formula and the assumption $\mathfrak{Ric}_{r,s}^E \geq -K$ imply
\[
\Delta|\phi|^2 \geq -2K|\phi|^2 + 2|\nabla\phi|^2.
\]
Combining this with \eqref{equ:parts}, we get
\begin{align*}
0 &\geq \int_{B\left(x_0,2\rho\right)} \left(-\chi^2 \Delta|\phi|^2 - 2K\chi^2|\phi|^2 + 2\chi^2|\nabla\phi|^2\right) \\
&\geq \int_{B\left(x_0,2\rho\right)} \left(\frac{1}{2}|\nabla\phi|^2 \chi^2 - \frac{8}{3}|\phi|^2|\nabla\chi|^2 - 2K\chi^2|\phi|^2\right).
\end{align*}
According to Lemma 6.8 of \cite{GM75} (see also \cite[Lemma 4.1]{EGHP23}), we have
\[
|\bar\partial\phi|^2 + |\bar\partial^*\phi|^2 \leq 2n|\nabla\phi|^2.
\]
Using the definition of $\chi$, we get
\begin{equation}\label{equ:2key}
\frac{1}{4n} \int_{B\left(x_0,\rho\right)} \left(|\bar\partial\phi|^2 + |\bar\partial^*\phi|^2\right) \leq \left(\frac{32}{3\rho^2} + 2K\right)\int_{B\left(x_0,2\rho\right)}  |\phi|^2.
\end{equation}

Combining \eqref{equ:moser 1} and \eqref{equ:2key} yields
\[
\sup_{B(x_0,\rho/2)} \left(|\bar\partial\phi|^2 + |\bar\partial^*\phi|^2\right) \leq C_{10} \sup_{B\left(x_0,2\rho\right)} |\phi|^2,
\]
where
\[
C_{10} =2^{3n^2+9n+3}\omega_{2n}\left(\frac{K_0\rho^2}{32} + 1\right)^{n}\left(\frac{16}{3\rho^2} + K\right)\leq \frac{C_{11}}{\rho^2}.
\]
The proof is complete.
\end{proof}

Now we can give the proof of Part (iii) of Theorem \ref{thm:Green form}.
\begin{proof}
Fix $x,y\in\overline\Omega$ with $x\neq y$. By continuity, we may assume $x,y\in\Omega$.
    We consider two cases.\\
    {\bf Case 1:} $\delta(y)\leq |x-y|$.\\
    \indent In this case, we know $G_{r,s}^E(x,\cdot)$ is  harmonic in $B\left(y,\frac{1}{2}\delta(y)\right)$. 
    By Lemma \ref{lem:important estimate},  
    $$|\bar\partial_y G_{r,s}^E(x,y)|+|\bar\partial_y^* G_{r,s}^E(x,y)|\leq \frac{\sqrt{32C_{11}}}{\delta(y)}\sup_{B\left(y,\frac{1}{2}\delta(y)\right)}|G(x,\cdot)|.$$
    For any $z\in B\left(y,\frac{1}{2}\delta(y)\right)$, 
    $$|x-z|\geq |x-y|-|y-z|\geq |x-y|-\frac{1}{2}\delta(y)\geq \frac{1}{2}|x-y|,$$
    $$\delta(z)\leq \delta(y)+|y-z|\leq 2\delta(y).$$ 
    By Part (ii) of Theorem \ref{thm:Green form},  
    \begin{align*}
    |G(x,z)|\leq C_2|x-z|^{1-2n}\delta(z)\leq 4^nC_2|x-y|^{1-2n}\delta(y).
    \end{align*}
    Therefore,
    \begin{equation}\label{gradient:eequu1}
    |\bar\partial_y G_{r,s}^E(x,y)|+|\bar\partial_y^* G_{r,s}^E(x,y)|\leq 4^nC_2\cdot\sqrt{32C_{11}}|x-y|^{1-2n}.
    \end{equation}
    {\bf Case 2:} $\delta(y)>|x-y|$. \\
    \indent In this case, $G_{r,s}^E(x,\cdot)$ is a harmonic function in $B\left(y,\frac{1}{2}|x-y|\right)$, so 
    $$|\bar\partial_y G_{r,s}^E(x,y)|+|\bar\partial_y^* G_{r,s}^E(x,y)|\leq \frac{\sqrt{32C_{11}}}{|x-y|}\sup_{B\left(y,\frac{1}{2}|x-y|\right)}|G(x,\cdot)|.$$
    For any $z\in B\left(y,\frac{1}{2}|x-y|\right)$, we have 
    $$|x-z|\geq |x-y|-|y-z|\geq \frac{1}{2}|x-y|,$$
    By Part (i) of Theorem \ref{thm:Green form}, 
    $$|G(x,z)|\leq C_1|x-z|^{2-2n}\leq 4^{n-1}C_1|x-y|^{2-2n}.$$
    Thus, 
    \begin{equation}\label{gradient:eequu2}
    |\bar\partial_y G_{r,s}^E(x,y)|+|\bar\partial_y^* G_{r,s}^E(x,y)|\leq 4^{n-1}C_1\cdot\sqrt{32C_{11}}|x-y|^{1-2n}.
    \end{equation}
    \indent Combining Inequalities (\ref{gradient:eequu1}) and (\ref{gradient:eequu2}), we obtain (iii).
\end{proof}

Now we state and prove a more general version of Theorem \ref{thm:section}.
\begin{thm}\label{thm:last}
Under the same assumptions and notations as in Theorem \ref{thm:Green form}. For any 
$1<p<\infty$, set 
$$p^*:=\frac{2np}{2n-1},\quad p^{\sharp}:=\frac{2np}{2n+p-1}.$$
Then for any $f\in C^1(\Omega,\Lambda^{r,s}T^*\mc^n\otimes E)\cap C^0(\overline\Omega,\Lambda^{r,s}T^*\mc^n\otimes E)$,
  there are constants $\delta_1,\delta_2>0$ such that 
$$\|f\|_{L^{p^*}(\Omega)}\leq \delta_1\|\bar\partial f\|_{L^{p^{\sharp}}(\Omega)}+
\delta_1\|\bar\partial^* f\|_{L^{p^{\sharp}}(\Omega)}+\delta_2\|f\|_{L^p(\partial\Omega)}.
$$
\end{thm}
\indent To prove Theorem \ref{thm:last}, we may assume $f\in C^2(\overline\Omega,\Lambda^{r,s}T^*\mc^n\otimes E)$.
By the Green representation formula (see \cite[Theorem 1.3]{DHQ25}),
\begin{align*}
    f&=\int_{\Omega}\langle \square_{r,s}^Ef,G_{r,s}^E\rangle dV+\int_{\partial \Omega}(\langle f,\bar\partial^*G_{r,s}^E\wedge \bar\partial\rho\rangle-\langle \bar\partial G_{r,s}^E,f\wedge \bar\partial\rho\rangle)\frac{dS}{|\nabla\rho|}\\
    &=\int_{\Omega}\left(\langle \bar\partial f,\bar\partial G_{r,s}^E\rangle+\langle \bar\partial^*f,\bar\partial^*G_{r,s}^E\rangle\right) dV\\
    &\quad +\int_{\partial \Omega}(\langle f,\bar\partial^*G_{r,s}^E\wedge \bar\partial\rho\rangle-\langle \bar\partial G_{r,s}^E,f\wedge \bar\partial\rho\rangle)\frac{dS}{|\nabla\rho|} \text{ in }\Omega,
    \end{align*} 
where $\rho$ is a smooth boundary defining function. By Minkowski inequality, it suffices to prove the following lemma:
\begin{lem}
The following estimates hold:
\begin{itemize}
  \item[(i)] \leavevmode For any $g \in L^{p^\sharp}(\Omega,\Lambda^{r,s}T^*\mc^n\otimes E)$,

  \[
  \|T_\Omega g\|_{L^{p^*}(\Omega)} \leq 2\omega_{2n}^{1-\frac{1}{2n}}
     \left(\frac{36^np}{p-1}\right)^{1-\frac{p}{2n}}
     \left(\frac{p-1}{n-2p}\right)^{\frac{p-1}{2n}}C_3 \|g\|_{L^{p^\sharp}(\Omega)},
  \]
  where
  \[
  T_\Omega g := \int_{\Omega} \langle \bar\partial g,\bar\partial G_{r,s}^E\rangle dV.
  \]
   \item[(ii)] \leavevmode For any $g \in L^{p^\sharp}(\Omega,\Lambda^{r,s}T^*\mc^n\otimes E)$,
  \[
  \|S_\Omega g\|_{L^{p^*}(\Omega)} \leq 2\omega_{2n}^{1-\frac{1}{2n}}
     \left(\frac{36^np}{p-1}\right)^{1-\frac{p}{2n}}
     \left(\frac{p-1}{n-2p}\right)^{\frac{p-1}{2n}}C_3 \|g\|_{L^{p^\sharp}(\Omega)},
  \]
  where
  \[
  S_\Omega g := \int_{\Omega} \langle \bar\partial^* g,\bar\partial^* G_{r,s}^E\rangle dV.
  \]

  \item[(iii)] \leavevmode For any $g\in L^\infty(\partial\Omega,\Lambda^{r,s}T^*\mc^n\otimes E)$,
  \[
  \|T_{\partial\Omega} g\|_{L^{p^*}(\Omega)} \leq 2(p-1)^{\frac{1-2n}{2np}}p^{-\frac{1}{2np}}
(4n\omega_{2n}^{1-\frac{1}{2n}}C_3)^{\frac{1}{p}}(e^{K\diam(\Omega)})^{1-\frac{1}{p}} \|g\|_{L^p(\partial\Omega)},
  \]
  where
  \[
  T_{\partial\Omega} g := \int_{\partial \Omega}(\langle g,\bar\partial^*G_{r,s}^E\wedge \bar\partial\rho\rangle-\langle \bar\partial G_{r,s}^E,g\wedge \bar\partial\rho\rangle)\frac{dS}{|\nabla\rho|}.
  \]
\end{itemize}
\end{lem}
\begin{proof}
(i) and (ii) follow directly from Lemma \ref{lem:better estimate} and Part (iii) of Theorem \ref{thm:Green form}.\\
(iii) For any $g\in  C^0(\partial\Omega,\Lambda^{r,s}T^*\mc^n\otimes E)$, we may find 
$u$ such that $\square_{r,s}^Eu=0$ in $\Omega$ and $u|_{\partial\Omega}=g$. By 
Lemma \ref{lem:Linfty estimate},
\begin{equation}\label{equ:continuous}
\|T_{\partial\Omega}g\|_{L^\infty(\Omega)}=\|u\|_{L^\infty(\Omega)}\leq e^{K\diam(\Omega)}\|g\|_{L^\infty(\partial\Omega)}.
\end{equation}
Now for a general $g\in L^\infty(\partial\Omega,\Lambda^{r,s}T^*\mc^n\otimes E)$, any $t>1$ and 
any $1<m<(2nt)/(2n-1)$, we can choose a sequence
$g_k\in C^0(\partial\Omega,\Lambda^{r,s}T^*\mc^n\otimes E)$ such that 
$$\|g_k\|_{L^\infty(\partial\Omega)}\leq \|g\|_{L^\infty(\partial\Omega)},\ \lim_{k\rw\infty}\|g_k-g\|_{L^t(\partial\Omega)}=0.$$
By Theorem \ref{thm:Green form}, it is straightforward to verify that (see the proof of \cite[Theorem 1.4]{DJQ24})
$$\lim_{k\rw\infty}\|T_{\partial\Omega}g_k-T_{\partial\Omega}g\|_{L^{m}(\Omega)}=0.$$
By Inequality (\ref{equ:continuous}) and Hölder's inequality, 
$$\|T_{\partial\Omega}g\|_{L^m(\Omega)}
\leq \limsup_{k\rw\infty}|\Omega|^{\frac{1}{m}}\|T_{\partial\Omega}g_k\|_{L^\infty(\Omega)}
\leq |\Omega|^{\frac{1}{m}}e^{K\diam(\Omega)}\|g\|_{L^\infty(\partial\Omega)}.$$
Let $t,m\rw \infty$, we obtain
\begin{equation}
\|T_{\partial\Omega}g\|_{L^\infty(\Omega)}\leq e^{K\diam(\Omega)} \|g\|_{L^\infty(\partial\Omega)}.
\end{equation}
Note that for any $(r,s)$-form $\alpha$, and any $(0,1)$-form $\beta$
$$|\alpha\wedge\beta|\leq |\alpha|\cdot|\beta|.$$
Using Lemma \ref{lem:weak}, Part (iii) of Theorem \ref{thm:Green form} and Cauchy-Schwarz inequality, we get
$$\|B_{\partial\Omega} g\|_{L^{\frac{2n}{2n-1},\infty}(\Omega)}\leq 2n\omega_{2n}^{1-\frac{1}{2n}}C_3\|g\|_{L^1(\Omega)},
\ \forall g\in L^1(\partial\Omega,\Lambda^{r,s}T^*\mc^n\otimes E).$$
The desired estimate now follows from Lemma \ref{lem:interpolation}.
\end{proof}

\bibliographystyle{alphanumeric}

\end{document}